\documentclass[conference]{IEEEtran}
\IEEEoverridecommandlockouts

\usepackage{cite}
\usepackage{amsmath,amssymb,amsfonts}
\usepackage{amsthm}
\usepackage{algorithmic}
\usepackage{graphicx}
\usepackage{textcomp}
\usepackage{xcolor}
\usepackage{stmaryrd}
\usepackage{bm}
\usepackage{tikz}
\usepackage{accents}
\usetikzlibrary{arrows, arrows.meta}
\usetikzlibrary{shapes.geometric, arrows}
\graphicspath{{images/}}
\usepackage{textcomp}

\newcommand\restr[2]{{
		\left.\kern-\nulldelimiterspace 
		#1 
		\littletaller 
		\right|_{#2} 
}}

\newcommand{\littletaller}{\mathchoice{\vphantom{\big|}}{}{}{}}

\newcommand\munderbar[1]{%
	\underaccent{\bar}{#1}}

\newtheorem{definition}{Definition}
\newtheorem{theorem}{Theorem}
\newtheorem{proposition}{Proposition}
\newtheorem{lemma}{Lemma}

\newtheorem{assumption}{Assumption}
\newtheorem{remark}{Remark}
\allowdisplaybreaks

\def\BibTeX{{\rm B\kern-.05em{\sc i\kern-.025em b}\kern-.08em
    T\kern-.1667em\lower.7ex\hbox{E}\kern-.125emX}}
\begin{document}

\title{Vertical Contracts for Safety Control
}

\author{\IEEEauthorblockN{Armin Pirastehzad and Bart Besselink}
\IEEEauthorblockA{\textit{Bernoulli Institute for Mathematics, Computer Science and Artificial Intelligence} \\
\textit{University of Groningen}\\
Groningen, The Netherlands \\
email: a.pirastehzad@rug.nl; b.besselink@rug.nl}
}

\maketitle

\begin{abstract}
	We propose a methodology that exploits the contract formalism to characterize the continuous-time safety control problem, which is often difficult to address, in terms of a discrete-time one, for which numerous efficient solution scheme exist. We construct contracts as pairs of assumptions and guarantees which are set-valued mappings that describe the safe boundaries within which the system must operate. By formalizing safety control as contract implementation, we develop a vertical hierarchy according to which we translate implementation from continuous to discrete time. We accomplish this by constructing a discrete-time system and a contract such that a solution to the  continuous-time implementation problem can be characterized in terms of a solution to its discrete-time counterpart. We then use this characterization to construct a control input that establishes implementation in continuous time on the basis of the control sequence that achieves implementation in discrete time.  
\end{abstract}


\section{Introduction}\label{Sec_Introduction}
	\IEEEPARstart{S}{afety} has become a crucial requisite in many engineering applications, as they require mobilization of systems whose malfunction may result in casualty and/or financial loss. It is due to such significance that safety guarantee often appears as an integral objective in control system design. Despite the development of many efficient safety control schemes for both continuous-time and discrete-time systems, guaranteeing safety for the former remains significantly more challenging than for the latter. This paper proposes a methodology for the translation of a continuous-time safety control problem, which is often difficult to address, into a discrete-time one, which can be efficiently addressed by the existing control schemes.
	
	As for continuous-time systems, safety control schemes (\textit{e.g.,} those proposed in \cite{7782377,8405547,8796030,XU2018195,8404080,9161270,9683085, 10383564,9992999,10421775,9147387,9483028, 10886042,9516971,COHEN2024100947,GARG2024100945,lindemann2025formal}) usually have in common that they formulate safety as forward invariance of a spatial set specified by (super)level sets of the so-called \emph{control barrier functions} \cite{GARG2024100945}. This leads to the characterization of invariance as the feasibility of a quadratic program, whose solution gives the control input that establishes invariance. Despite its elegance and efficiency, such an approach presents a few challenges. Namely, constraints on the input may severely hinder the construction of a \emph{valid} control barrier function, which jeopardizes the \emph{feasibility} of the quadratic program based on which the control input is obtained. Furthermore, since the solution of this quadratic program gives the control input at each instant of time, construction of the control input necessitates the solution of the quadratic program in a \emph{sampled-data} fashion. This results in evaluation of the control input at only specific instants of time rather than at all times. The implementation of such control, however, guarantees safety only at these specific instants, whereas it does not necessarily ensure safety in between them. 
	
	While safety control is challenging in continuous time, it is efficiently addressed in discrete time. For discrete-time systems, safety control schemes (\textit{e.g.,} those proposed in \cite{belta2019formal,7039363,10.1145/3575870.3587120,CARDONA2025101576,9769752,8310901,9351643,9261610,kouvaritakis2016model}) exploit \emph{sequentiality} of the discrete-time trajectories to formulate safety as constraints in an optimization problem that is completely in terms of the control input sequence. This allows for the utilization of efficient algorithms, such as those from convex optimization \cite{boyd2004convex}, to obtain the control sequence that achieves safety. 
	
	Motivated by the efficiency of safety control in discrete time, the goal of this paper is to propose a methodology for the formulation of a continuous-time safety control problem in terms of a discrete-time one. Such formulation then enables the derivation of a control input that establishes safety in continuous time based on the control sequence that achieves safety in discrete time. The main contributions of this paper are listed as follows. 
	
	First, we adopt the \emph{contract} formalism \cite{EDA-053} to develop a unifying framework for expressing safety specifications for continuous-time and discrete-time systems. We construct contracts as pairs of set-valued mappings, called \emph{assumptions} and \emph{guarantees}, that specify safety of the input-state trajectories. Specifically, the assumptions indicate restrictions/limitations that the control input must satisfy, whereas the guarantees describe safe boundaries within which the system must operate. We then formalize the safety control problem within this framework as \emph{contract implementation}, which concerns the feasibility of enforcing safety (prescribed by the guarantees) despite the design limitations (indicated by the assumptions). 
	
	Second, we establish contract implementation according to a \emph{vertical} hierarchy, where we translate contract implementation from continuous to discrete time such that a solution of the continuous-time problem can be characterized by that of the discrete-time one. For some suitable sampling time, we discretize the continuous-time contract to construct a discrete-time one. We then discretize the continuous-time system (with respect to the same sampling time) into a discrete-time model according to the notion of system interpolation \cite{Pirastzehzad2025}, which allows for the realization of input-state trajectories of the continuous-time system as piecewise polynomial interpolations of input-state trajectories of the discrete-time model. Subsequently, we use such realization to derive a condition that guarantees implementation in continuous time (which corresponds to the continuous-time system and contract) on the basis of implementation in discrete time (which corresponds to the discrete-time model and the discretized contract). Lastly, given a control sequence that establishes implementation in discrete time, we construct a control input that yields implementation in continuous time. 
	
	The rest of this paper is organized as follows. In Section~\ref{Sec_Contracts}, we introduce contracts and formalize safety control as contract implementation. In Section~\ref{Sec_ProblemStatement}, we develop the vertical hierarchy for contract implementation. In Section~\ref{Sec_Discretization}, we propose discretization techniques for the continuous-time contracts and systems. We then use these techniques to translate contract implementation from continuous to discrete time in Section~\ref{Sec_Translation}. We illustrate these results in a numerical example in Section~\ref{Sec_NumericalExample} and, finally, conclude the paper by Section~\ref{Sec_Conclusion}.  

\subsubsection*{Notation}\label{Sec_Notation}
The set of real numbers and integers are denoted by $\mathbb{R}$ and $\mathbb{Z}$. Additionally, the closed real interval $\{t\in\mathbb{R} \vert t_1\leq t \leq t_2\}$ and the closed integer interval $\{k\in\mathbb{Z} \vert k_1\leq k \leq k_2\}$ are indicated by $[t_1,t_2]_{\mathbb{R}}$ and $[k_1,k_2]_{\mathbb{Z}}$, respectively. We denote the set of all subsets of a set $X\subset\mathbb{R}^n$ (also known as the power set of $X$) by $\operatorname{P}(X)$. We also denote by $\operatorname{conv}(X)$ the set of all convex subsets of $X$, \textit{i.e.,} $\operatorname{conv}(X) = \{C\subset X \vert \forall c_1,c_2\in C, \forall \alpha\in[0,1]_{\mathbb{R}}: \ \alpha c_1 + (1-\alpha) c_2 \in C\}$. We define the vectorization of a matrix $M =[M_1 \, M_2\,\cdots\,M_m]$, with columns $M_1,M_2,\cdots,M_m\in\mathbb{R}^n$, as $\operatorname{vec}(M) = (M_1^\top,M_2^\top,\cdots,M_m^\top)^\top$. We also define the operator $\operatorname{col}(\cdot)$ such that $\operatorname{col}(M) = \{M_1,M_2,\cdots,M_m\}$. The symbol $\otimes$ is utilized to denote the Kronecker product. Indicating the identity and zero matrices of appropriate dimensions by $I$ and $0$, the Kronecker sum of square matrices $N$ and $M$ is defined as $N\oplus M = N\otimes I + I\otimes M$. For sets $V$ and $X$,  we denote the set of all  set-valued mappings $\pi: V \rightrightarrows \operatorname{conv}(X)$ by $\Pi_V^X$. For any $\pi \in \Pi_V^X$, we denote by $\operatorname{sig}(\pi)$ the set of all functions $s: V\rightarrow X$ such that $s(v) \in \pi(v)$ for all $v\in V$. We use $\mathbb{P}_N^n$ to denote the space of all $n$-dimensional vector-valued polynomials up to degree $N$. Given sets $V$ and $X$, we define the restriction of a function $f : V \rightarrow X$ to the subset $V'\subset V$ as the function $\restr{f}{V'} : V' \rightarrow X$ such that $\forall v \in V' : \restr{f}{V'} (v) = f(v)$. We also define the restriction of a set $\mathcal{F}$ of functions $f:V \rightarrow X$ to $V'$ as $\restr{\mathcal{F}}{V'} = \left\{\restr{f}{V'} \big\vert f\in\mathcal{F}\right\}$.  
\section{Safety Contracts}\label{Sec_Contracts}
We develop a formal framework for expressing safety specifications in both continuous-time and discrete-time settings. In continuous time, we consider the system
\begin{equation}\label{ContinuousSystem}
	\bm{\Sigma}_{c}: \begin{aligned}
		\dot{x}_{c}(t) &= A_{c}x_{c}(t) + B_{c}u_{c}(t),
	\end{aligned} 
\end{equation}
with state $x_{c}\in\mathbb{R}^n$ and control input $u_{c} \in \mathbb{R}^m$. We denote by $x_{c}(t;x_{0},u_{c})$ the state solution, at time $t$, of \eqref{ContinuousSystem} for initial condition $x_{c}(0) = x_{0}$ and input $u_c$.

Correspondingly, in discrete time, we consider the system
\begin{equation}\label{DiscreteSystem}
	\bm{\Sigma}_{d}: \begin{aligned}
		x_{d}(k+1) &= A_{d}x_{d}(k) + B_{d}u_{d}(k),
	\end{aligned} 
\end{equation}
with state $x_{d}\in\mathbb{R}^n$ and $u_{d}\in\mathbb{R}^m$. We use notation similar to that of \eqref{ContinuousSystem} to denote the state sequence of \eqref{DiscreteSystem}. 

We adopt the contract formalism \cite{EDA-053} to express \emph{safety} for the continuous-time system \eqref{ContinuousSystem} and its discrete-time counterpart \eqref{DiscreteSystem}. To treat both cases uniformly, we define $\mathbb{S}\in\{\mathbb{R},\mathbb{Z}\}$ as a generic time domain and introduce the symbol $s\in\{c,d\}$ as a generic label that indexes continuous-time or discrete-time systems. For a given terminal time $S\in\mathbb{S}$, we specify the \emph{safe} input-state behavior over the time interval $[0,S]_{\mathbb{S}}$ by introducing \emph{assumptions} $\mathcal{A}_{s}\in\Pi_{[0,S]_{\mathbb{S}}}^{\mathbb{R}^m}$ and \emph{guarantees} $\mathcal{G}_{s}\in\Pi_{[0,S]_{\mathbb{S}}}^{\mathbb{R}^n}$. Based on these two set-valued mappings, we define contracts as follows. 
\begin{definition}\label{Def_Contracts}
	A contract $\mathcal{C}_{s}$ is a pair $(\mathcal{A}_{s},\mathcal{G}_{s})$ of assumptions $\mathcal{A}_{s}$ and guarantees $\mathcal{G}_{s}$. 
\end{definition}

We regard a contract $\mathcal{C}_{s}$ as a formal safety specification on the input-state behavior of $\bm{\Sigma}_s$. 
\begin{definition}\label{Def_Implementation}
	Given a contract $\mathcal{C}_{s} = (\mathcal{A}_{s},\mathcal{G}_{s})$, for an initial condition $x_0\in\mathbb{R}^n$, the system $\bm{\Sigma}_{s}$ implements the contract $\mathcal{C}_{s}$, denoted by $\bm{\Sigma}_{s} \models \mathcal{C}_{s}$, if there exists a control input $u_{s}\in\operatorname{sig}(\mathcal{A}_{s})$ such that $\restr{x_{s}(\cdot; x_0,u_{s})}{[0,S]_{\mathbb{S}}} \in \operatorname{sig}(\mathcal{G}_{s})$.
\end{definition}

The contract $\mathcal{C}_{s}$ specifies the safety of input-state trajectories over the time interval $[0,S]_{\mathbb{S}}$. It follows from Definition~\ref{Def_Implementation} that over the time interval $[0,S]_{\mathbb{S}}$, the contract $\mathcal{C}_{s}$ requires the control input $u_{s}$ to never exceed the \emph{limitations} imposed by assumptions $\mathcal{A}_{s}$, while it expects the state solution $x_{s}$ to remain within the \emph{safe boundaries} prescribed by guarantees $\mathcal{G}_{s}$. We therefore observe that $\bm{\Sigma}_{s}\models\mathcal{C}_{s}$ basically ensures the enforcement of safety (as prescribed by $\mathcal{G}_{s}$) despite the design limitations (as dictated by $\mathcal{A}_{s}$).
\section{Vertical Contract implementation}\label{Sec_ProblemStatement}
Consider a continuous-time system $\bm{\Sigma}_c$ and let a contract $\mathcal{C}_c = (\mathcal{A}_c,\mathcal{G}_c)$ be such that $\mathcal{A}_c\in \Pi_{[0,T]_{\mathbb{R}}}^{\mathbb{R}^m}$ and $\mathcal{G}_c\in\Pi_{[0,T]_{\mathbb{R}}}^{\mathbb{R}^n}$ for some $T>0$. For an initial condition $x_0\in\mathbb{R}^n$, our objective is to establish $\bm{\Sigma}_c \models \mathcal{C}_c$. This, by Definition~\ref{Def_Implementation}, reduces to the construction of a continuous-time control input $u_c\in\operatorname{sig}(\mathcal{A}_c)$ such that 
\begin{equation}\label{ContinuousTimeCondition}
	\restr{x_c(\cdot;x_0,u_c)}{[0,T]_{\mathbb{R}}} \in \operatorname{sig}(\mathcal{G}_c). 
\end{equation}
However, construction of such control input is not necessarily possible for a general contract $\mathcal{C}_c$. In fact, we note from the continuity of input-state trajectories in $\bm{\Sigma}_c$ that establishment of \eqref{ContinuousTimeCondition} necessitates a degree of `smoothness' in the behavior of set-valued mappings $\mathcal{A}_c$ and $\mathcal{G}_c$ with respect to time. To ensure the feasibility of this problem, we restrict our attention to contracts that are \emph{smooth} in the following sense. 
\begin{assumption}\label{Assumption}
	For all $t\in[0,T]_{\mathbb{R}}$, there exist constants $\munderbar{r}_{t},\bar{r}_t>0$ such that 
	\begin{equation}\label{Smoothness}
		\bigcap\limits_{\hat{t} \in [t-\munderbar{r}_t,t+\bar{r}_t]_{\mathbb{R}}} \mathcal{A}_c(\hat{t})\neq \emptyset, \quad \bigcap\limits_{\hat{t} \in [t-\munderbar{r}_t,t+\bar{r}_t]_{\mathbb{R}}} \mathcal{G}_c(\hat{t})\neq \emptyset
	\end{equation}
	where the intervals are truncated at $t=0$ and $t = T$.
\end{assumption}
Assumption~\ref{Assumption} basically ensures that the sets prescribed by $\mathcal{A}_c$ and $\mathcal{G}_c$ do not abruptly change in time, \textit{i.e.,} they \emph{smoothly} vary over time. This is a mild assumption, as many safety specifications in practice satisfy the property \eqref{Smoothness}.

With this characterization of continuous-time contracts, we now give a formal statement of the problem we aim to solve. 
\paragraph*{Continuous-time Contract Implementation} Given a continuous-time system $\bm{\Sigma}_c$ and a contract $\mathcal{C}_c = (\mathcal{A}_c,\mathcal{G}_c)$ that satisfies Assumption~\ref{Assumption}, for an initial condition $x_0\in\mathbb{R}^n$, find a continuous-time control input $u_c\in\operatorname{sig}(\mathcal{A}_c)$ that achieves \eqref{ContinuousTimeCondition}.

While contract implementation is challenging in continuous time, it can be efficiently addressed by numerous control schemes (\textit{e.g.,} model predictive control (MPC) \cite{7039363} and temporal logic control \cite{6907641,belta2019formal}) in discrete time. Motivated by this, we propose a methodology for the \emph{translation} of contract implementation from continuous to discrete time such that a solution to the continuous-time problem can be characterized by that of its discrete-time counterpart. 

For some suitable sampling time $\tau>0$, we will first discretize $\mathcal{C}_c$ to construct the discrete-time contract $\mathcal{C}_d = (\mathcal{A}_d,\mathcal{G}_d)$ such that $\mathcal{A}_d \in\Pi_{[0,\ell]_{\mathbb{Z}}}^{\mathbb{R}^m}$ and $\mathcal{G}_d \in\Pi_{[0,\ell]_{\mathbb{Z}}}^{\mathbb{R}^n}$ for some integer $\ell>0$. Then, with respect to the same sampling time, we will discretize $\bm{\Sigma}_c$ to construct the discrete-time system $\bm{\Sigma}_d$ in such a way that the following implication holds: 
\begin{equation}\label{ContractImplication}
	\bm{\Sigma}_d\models\mathcal{C}_d \Longrightarrow \bm{\Sigma}_c\models\mathcal{C}_c.
\end{equation}
At last, for a suitable $u_d\in\operatorname{sig}(\mathcal{A}_d)$ that achieves $\restr{x_d(\cdot;x_0,u_d)}{[0,\ell]_{\mathbb{Z}}} \in \operatorname{sig}(\mathcal{G}_d)$, we will exploit the implication \eqref{ContractImplication} to construct a $u_c\in\operatorname{sig}(\mathcal{A}_c)$ that yields \eqref{ContinuousTimeCondition}.

We will therefore establish contract implementation according to a \emph{vertical} hierarchy, where we formulate  continuous-time implementation (which is a difficult problem to solve) in terms of discrete-time implementation (which can be efficiently addressed).
%
%
%
%
%
%
\section{Discretization}\label{Sec_Discretization}
As the first step towards vertical contract implementation, we propose discretization techniques for the continuous-time contract $\mathcal{C}_c$ and the system $\bm{\Sigma}_c$.  
\subsection{Contract Discretization}\label{Sec_ContractDiscretization}
Consider a continuous-time contract $\mathcal{C}_c = (\mathcal{A}_c,\mathcal{G}_c)$ such that Assumption~\ref{Assumption} holds. We discretize $\mathcal{C}_c$ by sampling the assumptions $\mathcal{A}_c$ and the guarantees $\mathcal{G}_c$, which leads to the construction of a discrete-time contract $\mathcal{C}_d = (\mathcal{A}_d,\mathcal{G}_d)$. Considering that the ultimate goal is to characterize the implementation of $\mathcal{C}_c$ in terms of $\mathcal{C}_d$, we must sample the assumptions $\mathcal{A}_c$ and the guarantees $\mathcal{G}_c$ for sufficiently many times such that no crucial information is lost through sampling and that $\mathcal{C}_d$ captures all the essential aspects of $\mathcal{C}_c$. 

We select a sampling time $\tau>0$ sufficiently small to avoid loss of crucial information during sampling. For this purpose, we recall from Assumption~\ref{Assumption} that for all $t\in[0,T]_{\mathbb{R}}$, there exist constants $\munderbar{r}_t,\bar{r}_t>0$ such that \eqref{Smoothness} holds. We then define $r_c =  \min_{t\in[0,T-\munderbar{r}_T]_{\mathbb{R}}} \bar{r}_t$ and note that such constant $r_c$ exists as the minimum is taken over a closed time interval. We also emphasize that $r_c\leq \munderbar{r}_T$, which results from the fact that $\bar{r}_{T-\munderbar{r}_T} \leq \munderbar{r}_T$. We let $\lceil \frac{T}{r_c}\rceil$ denote the least integer larger than or equal to $\frac{T}{r_c}$. Accordingly, by selecting any integer $\ell_d\geq \lceil \frac{T}{r_c}\rceil$, we define the sampling time $\tau$ as
\begin{equation}\label{SamplingTime}
	\tau = \frac{T}{\ell_d}.
\end{equation}

We now sample the assumptions $\mathcal{A}_c$ and the guarantees $\mathcal{G}_c$ with respect to the sampling time \eqref{SamplingTime}. We thus construct the discrete-time set-valued mappings $\mathcal{A}_d: [0,\ell_d]_{\mathbb{Z}} \rightrightarrows \operatorname{P}(\mathbb{R}^m)$ and $\mathcal{G}_d: [0,\ell_d]_{\mathbb{Z}} \rightrightarrows \operatorname{P}(\mathbb{R}^n)$ such that 
\begin{subequations}\label{SampledContract}
	\begin{equation}\label{MidCondition}
		\begin{aligned}
			\mathcal{A}_d(k)\! = \!\!\!\bigcap\limits_{t \in [k\tau,(k+1)\tau]_{\mathbb{R}}} \!\!\!\mathcal{A}_c(t),\;\; \
			\mathcal{G}_d(k)\!= \!\!\!\bigcap\limits_{t \in [k\tau,(k+1)\tau]_{\mathbb{R}}} \!\!\!\mathcal{G}_c(t),
		\end{aligned}
	\end{equation}
	\text{for all $k\in [0,\ell_d-1]_{\mathbb{Z}}$ and that }
	\begin{equation}\label{TerminalCondition}
		\begin{aligned}
			\mathcal{A}_d(\ell_d) = \mathcal{A}_c(T),\quad \mathcal{G}_d(\ell_d) = \mathcal{G}_c(T).
		\end{aligned}
	\end{equation}
\end{subequations}
We note from \eqref{SamplingTime} that $\tau \leq r_c$. Given that $r_c\leq \munderbar{r}_T$, it readily follows from \eqref{Smoothness} that $\mathcal{A}_d(k) \neq \emptyset$ and $\mathcal{G}_d(k) \neq \emptyset$ for all $k\in [0,\ell_d]_{\mathbb{Z}}$. We then recall that $\mathcal{A}_c\in\Pi_{[0,T]_{\mathbb{R}}}^{\mathbb{R}^m}$ and $\mathcal{G}_c\in\Pi_{[0,T]_{\mathbb{R}}}^{\mathbb{R}^n}$, which implies that $\mathcal{A}_c(t) \in \operatorname{conv}(\mathbb{R}^m)$ and $\mathcal{G}_c(t) \in \operatorname{conv}(\mathbb{R}^n)$ for all $t\in[0,T]_{\mathbb{R}}$, \textit{i.e.,} $\mathcal{A}_c(t)$ and $\mathcal{G}_c(t)$ are convex sets. Since the intersection of any collection of convex sets is also convex, it directly follows from \eqref{SampledContract} that $\mathcal{A}_d(k)\in \operatorname{conv}(\mathbb{R}^m)$ and $\mathcal{G}_d(k)\in \operatorname{conv}(\mathbb{R}^n)$ for all $k\in [0,\ell_d]_{\mathbb{Z}}$. This indicates that $\mathcal{A}_d \in \Pi_{[0,\ell_d]_{\mathbb{R}}}^{\mathbb{R}^m}$ and $\mathcal{G}_d \in \Pi_{[0,\ell_d]_{\mathbb{R}}}^{\mathbb{R}^n}$. We finally collect $\mathcal{A}_d$ and $\mathcal{G}_d$ into $\mathcal{C}_d = (\mathcal{A}_d,\mathcal{G}_d)$ and immediately conclude that $\mathcal{C}_d$ is a discrete-time contract according to the Definition~\ref{Def_Contracts}.

We observe that the discrete-time contract $\mathcal{C}_d$ constructed as above encodes the safety specifications prescribed by the continuous-time contract $\mathcal{C}_c$. In fact, it follows from the non-emptiness of discrete-time assumptions $\mathcal{A}_d$ and guarantees $\mathcal{G}_d$ (at each instant $k$) that the safety specifications described by continuous-time assumptions $\mathcal{A}_c$ and guarantees $\mathcal{G}_c$ over each time interval $[k\tau,(k+1)\tau]_{\mathbb{R}}$ are respectively captured by the safety specifications indicated by $\mathcal{A}_d$ and $\mathcal{G}_d$ at instant $k$. 
\subsection{System Discretization}\label{Sec_SysDiscretization}
We discretize $\bm{\Sigma}_c$ according to the notion of system interpolation \cite{Pirastzehzad2025}, which is a comparative framework for systems from different time domains. 

Towards a formal definition of the notion of system interpolation, we first give the following interpretation of piecewise polynomial interpolation. 
\begin{definition}\label{DefSignalInterpolation}
	For integers $\ell,p>0$, let a discrete-time signal $s_d: [0,\ell]_{\mathbb{Z}} \rightarrow \mathbb{R}^p$ be given. For a sampling time $\tau>0$ and an integer $N\geq 1$, a continuous-time signal $s_c: [0,\ell \tau]_{\mathbb{R}}\rightarrow\mathbb{R}^p$ is said to be an $N$-th order interpolation of $s_d$ with respect to the sampling time $\tau$ if there exist polynomial functions $P_0,P_1,\cdots,P_{\ell-1}\in\restr{\mathbb{P}_N^{p}}{[0,\tau]_{\mathbb{R}}}$ such that for all $k \in [0,\ell-1]_{\mathbb{Z}}$, we have that $s_c(k\tau) = s_d(k)$ and $s_c\big((k+1)\tau\big) = s_d(k+1)$, while $s_c(k\tau+t) = P_k(t)$ for all $t\in[0,\tau]_{\mathbb{R}}$.  
\end{definition}
We denote by $\mathbb{I}_{N}^\tau(s_d)$ the set of all $N$-th order interpolations of $s_d$ with respect to $\tau$. Based on this,  we give a formal definition of system interpolation, which is taken from \cite{Pirastzehzad2025}.
\begin{definition}\label{DefInterpolatingSystem}
	For a sampling time $\tau>0$ and an integer $N\geq 1$, a continuous-time system $\bm{\Sigma}_c$ is said to be an $N$-th order interpolator of a discrete-time system $\bm{\Sigma}_d$ with respect to the sampling time $\tau$ if for any integer $\ell>0$, any initial condition $x_0\in\mathbb{R}^n$, and any discrete-time input $u_d:[0,\ell]_{\mathbb{Z}} \rightarrow\mathbb{R}^m$, there exists a continuous-time input $u_c \in \mathbb{I}_N^\tau (u_d)$ such that 
	\begin{equation}\label{InterpolatingCondition}
		\restr{x_c (\cdot; x_0,u_c)}{[0,\ell\tau]_{\mathbb{R}}} \in \mathbb{I}_N^\tau \Big(\restr{x_d(\cdot;x_0,u_d)}{[0,\ell]_{\mathbb{Z}}}\Big).
	\end{equation}
\end{definition}
It follows from Definition~\ref{DefInterpolatingSystem} that $\bm{\Sigma}_c$ is able to generate an $N$-th order interpolation of any input-state trajectory of $\bm{\Sigma}_d$. 

We obtain an algebraic characterization of system interpolation on the basis of the so-called shifted Legendre polynomials (for further details, see, \textit{e.g.,} \cite{Canuto2006}).

 Given any integer $k\geq 0$, let $\mathfrak{L}_k$ be the \emph{standard} Legendre polynomial of degree $k$ (see, \textit{e.g.,} \cite{Canuto2006}). Then, for $\tau>0$, the \emph{shifted} Legendre polynomial of degree $k$, over the interval $[0,\tau]_{\mathbb{R}}$, is defined as $\mathfrak{L}_k^\tau (t) = \mathfrak{L}_k(\frac{2t}{\tau} -1)$. For an integer $N\geq 1$, we let the nodes $0 = t_{N,0}<t_{N,1}<\ldots<t_{N,N} < \tau$ be such that $\mathfrak{L}_N^\tau(t_{N,i}) + \mathfrak{L}_{N+1}^\tau(t_{N,i}) = 0$ for all $i=0,1,\ldots,N$. We note that such nodes always exist as $\mathfrak{L}_N^\tau + \mathfrak{L}_{N+1}^\tau$ has $N+1$ distinct roots (see \cite[Section 2.2.3]{Canuto2006} for further details). For each $i=0,1,\ldots,N$, these nodes define a function
	\begin{equation}\label{PhiFunction}
		\phi_i(t) = \prod_{j=0,j\neq i}^{N}\frac{t-t_{N,j}}{t_{N,i}-t_{N,j}},
	\end{equation}
	which is an $N$-th order polynomial such that $\phi_i(t_{N,i}) = 1$, whereas $\phi_i(t_{N,j}) = 0$ for all $j\neq i$. Based on these polynomials, we define the vector-valued polynomial function $\Phi:[0,\tau]_{\mathbb{R}}\rightarrow\mathbb{R}^{N}$ such that 
	\begin{equation*}
		\forall t\in[0,\tau]_{\mathbb{R}}: \quad \Phi(t) = \begingroup 
		\setlength\arraycolsep{1.5pt}\begin{bmatrix}
			\phi_1(t)&\cdots&\phi_N(t)
		\end{bmatrix}\endgroup^\top,
	\end{equation*}
	and, accordingly, we construct 
	\begin{equation*}
		\begin{aligned}
			\sigma = \begingroup 
			\setlength\arraycolsep{1.25pt}\begin{bmatrix}
				\dot{\phi}_0(t_{N,1})&\cdots&\dot{\phi}_0(t_{N,N})
			\end{bmatrix}\endgroup, \; \Psi = \begingroup 
			\setlength\arraycolsep{2.85pt}\begin{bmatrix}
				\dot{\Phi}(t_{N,1})&\cdots&	\dot{\Phi}(t_{N,N})
			\end{bmatrix}\endgroup\!.
		\end{aligned}
	\end{equation*}
	Finally, by defining the matrices 
			\begin{align}
				Q_N^\tau (A_c,B_c) \!&= \! \begingroup\setlength\arraycolsep{0pt}\begin{bmatrix}
					\dot{\Phi}^\top(0) \otimes I & 0\\
					\Psi^\top\oplus(-A_c)&-I \otimes B_c\\
					\Phi^\top(\tau) \otimes I&0\\
					0 &\Phi^\top(\tau) \otimes I
				\end{bmatrix}\endgroup\!\!, \; T_1= \!\begin{bmatrix}
				0\\0\\I\\0
				\end{bmatrix}\!\!,\label{QRSV}\\
				R_N^\tau(A_c,B_c)\! &= \!\begingroup\setlength\arraycolsep{1.25pt}\begin{bmatrix}
					A_c-\dot{\phi}_0(0)I&B_c&0\\
					-\sigma^\top\otimes I&0&0\\
					- \phi_0(\tau)I&0&0\\
					0&\!\!\!-\phi_0(\tau)I&I
				\end{bmatrix}\endgroup\!\!,\; \
				T_2 = \begingroup\setlength\arraycolsep{2pt}\begin{bmatrix}
					I&0&0\\
					0&I&0
				\end{bmatrix}\endgroup,\nonumber
			\end{align}
we obtain the following characterization of system interpolation, which is a restatement of Theorem 1 in \cite{Pirastzehzad2025}.
\begin{proposition}\label{PropCharacterization}
	For a sampling time $\tau>0$ and an integer $N\geq 1$, $\bm{\Sigma}_c$ is an $N$-th order interpolator of $\bm{\Sigma}_d$ with respect to the sampling time $\tau$ if and only if 
	\begin{equation*}
		\begin{aligned}
			\operatorname{Im}\Big(R_N^ \tau(A_c,B_c) + T_1\begin{bmatrix}
				A_d&B_d
			\end{bmatrix} T_2\Big)\subset
			\operatorname{Im}Q_N^\tau (A_c,B_c).
		\end{aligned}
	\end{equation*} 
\end{proposition}

Proposition~\ref{PropCharacterization} facilitates the characterization of all input-state trajectories that establish \eqref{InterpolatingCondition}. Suppose that $\bm{\Sigma}_c$ is an $N$-th order interpolator of $\bm{\Sigma}_d$ with respect to the sampling time $\tau$. Given an integer $\ell>0$, an initial condition $x_0\in\mathbb{R}^n$, and a control sequence $u_d:[0,\ell]_{\mathbb{Z}}\rightarrow\mathbb{R}^m$, for any integer $0\leq k\leq \ell-1$, it follows from Proposition~\ref{PropCharacterization} that there exist matrices $X^k\in\mathbb{R}^{n\times N}$ and $U^k\in\mathbb{R}^{m\times N}$ such that 
\begin{equation}\label{AlgorithmTechnical1}
	\begin{aligned}
		&Q_N^\tau (A_c,B_c) \begin{bmatrix}
			\operatorname{vec}(X^k)\\\operatorname{vec}(U^k)
		\end{bmatrix}\\
		& = \Big(R_N^ \tau(A_c,B_c) + T_1\begingroup 
		\setlength\arraycolsep{1pt}\begin{bmatrix}
			A_d&B_d
		\end{bmatrix}\endgroup T_2\Big) \begin{bmatrix}
			x_d(k;x_0,u_d)\\u_d(k)\\u_d(k+1)
		\end{bmatrix}.
	\end{aligned}
\end{equation}
Correspondingly, we define signals $\bar{u}_c^k\in\restr{\mathbb{P}_N^m}{[0,\tau]_{\mathbb{R}}}$ and $\bar{x}_c^k\in\restr{\mathbb{P}_N^n}{[0,\tau]_{\mathbb{R}}}$ such that for all $t\in[0,\tau]_{\mathbb{R}}$, 
\begin{equation}\label{AlgorithmTechnical2}
	\begin{aligned}
		\bar{u}_c^k(t) &= 
			u_d(k)\phi_0(t) + U^k\Phi(t),\\
		\bar{x}_c^k(t) &= 
			x_d(k;x_0,u_d)\phi_0(t)+X^k\Phi(t).
	\end{aligned}
\end{equation}
Considering that the matrices $U^k$ and $X^k$ are not necessarily unique, we collect all signals $\bar{u}_c^k$ and $\bar{x}_c^k$ defined by \eqref{AlgorithmTechnical2} in the set 
	\begin{align}
		\mathcal{T}_c^k(x_0,u_d) = \Big\{\!(\bar{u}_c^k,\bar{x}_c^k) \Big\vert  \bar{u}_c^k\in&\restr{\mathbb{P}_N^m}{[0,\tau]_{\mathbb{R}}}, \bar{x}_c^k\in \restr{\mathbb{P}_N^n}{[0,\tau]_{\mathbb{R}}}, \label{SetTci}\\
		&\quad \text{$\exists U^k,\!X^k$s.t. \eqref{AlgorithmTechnical1}--\eqref{AlgorithmTechnical2} hold}\! \Big\}.\nonumber
	\end{align}
Then, for each $k=0,1,\ldots,\ell-1$, we let $(\bar{u}_c^k,\bar{x}_c^k)\in\mathcal{T}_c^k(x_0,u_d)$ and define the continuous-time signals $\bar{u}_c:[0,\ell \tau]_{\mathbb{R}}\rightarrow\mathbb{R}^m$ and $\bar{x}_c:[0,\ell \tau]_{\mathbb{R}}\rightarrow\mathbb{R}^n$ such that 
\begin{equation}\label{Stepwise}
	\begin{aligned}
		\forall t \in [0,\tau]: \;  \bar{u}_c(k\tau + t) = \bar{u}_c^k(t),\; \bar{x}_c(k\tau + t) = \bar{x}_c^k(t).
	\end{aligned}
\end{equation}
We finally collect all these continuous-time signals in the set 
\begin{align}
	&\mathcal{T}_c(x_0,u_d) = \Big\{\!(\bar{u}_c,\bar{x}_c) \big\vert\bar{u}_c\!:\![0,\ell \tau]_{\mathbb{R}}\rightarrow\mathbb{R}^m\!,\bar{x}_c\!:\![0,\ell \tau]_{\mathbb{R}}\rightarrow\mathbb{R}^n,\nonumber\\
	&\forall k\!\in\![0,\ell-1]_{\mathbb{Z}}, \text{$\exists (\bar{u}_c^k,\bar{x}_c^k)\!\in\!\mathcal{T}_c^k(x_0,u_d)\!$ s.t. \!\eqref{Stepwise} holds}\!\Big\},\label{SetTc}
\end{align}
based on which we obtain the following result, which is an immediate consequence of \cite[Theorem 2]{Pirastzehzad2025}. 
\begin{theorem}\label{PropTrajectorySetCharacterization}
	For a sampling time $\tau>0$ and an integer $N\geq 1$, suppose that $\bm{\Sigma}_c$ is an $N$-th order interpolator of $\bm{\Sigma}_d$ with respect to the sampling time $\tau$. Then, for any integer $\ell>0$, any initial condition $x_0\in\mathbb{R}^n$, and any discrete-time input $u_d:[0,\ell]_{\mathbb{Z}} \rightarrow \mathbb{R}^m$, a continuous-time input $u_c\in\mathbb{I}_N^\tau (u_d)$ establishes \eqref{InterpolatingCondition} if and only if $\big(\restr{u_c,x_c(\cdot;x_0,u_c)}{[0,\ell \tau]_{\mathbb{R}}}\big) \in \mathcal{T}_c(x_0,u_d)$.
\end{theorem}


Having characterized system interpolation and the set of all input-state trajectories that establish \eqref{InterpolatingCondition}, we now conduct system discretization (see \cite[Theorem 3]{Pirastzehzad2025} for further details).
\begin{theorem}\label{Prop_Design}
	Given a continuous-time system $\bm{\Sigma}_c$, for a sampling time $\tau>0$ and an integer $N\geq 1$, there exist matrices $A_d$ and $B_d$ such that $\bm{\Sigma}_c$ is an $N$-th order interpolator of $\bm{\Sigma}_d$ with respect to the sampling time $\tau$ if and only if there exist vectors $v_1$ and $v_2$ such that 
	\begin{equation*}
		\begin{bmatrix}
			I\otimes Q_N^\tau (A_c,B_c) &-T_2^\top \otimes T_1
		\end{bmatrix}\begin{bmatrix}
			v_1\\
			v_2\\
		\end{bmatrix} = \operatorname{vec}\big( R_N^\tau(A_c,B_c) \big),
	\end{equation*}
	where $Q_N^\tau (A_c,B_c)$, $R_N^\tau(A_c,B_c)$, $T_1$, and $T_2$ are given by \eqref{QRSV}. For any such vectors $v_1$ and $v_2$, the matrices $A_d$ and $B_d$ are recovered from $v_1 = \operatorname{vec}([A_d \; B_d])$.
\end{theorem}
Theorem~\ref{Prop_Design} allows us to construct the discrete-system $\bm{\Sigma}_d$ by solving a linear matrix equation, whose solution gives the matrices $A_d$ and $B_d$. 
\section{From Continuous-Time to Discrete-Time Implementation}\label{Sec_Translation}
We now employ the discretization techniques discussed in the previous section to translate the contract implementation problem from continuous to discrete time. Given a continuous-time system $\bm{\Sigma}_c$ and a contract $\mathcal{C}_c$ that satisfies Assumption~ \ref{Assumption}, we select the sampling time $\tau$ as in \eqref{SamplingTime} and, accordingly, construct the discrete-time contract $\mathcal{C}_d = (\mathcal{A}_d,\mathcal{G}_d)$ as in \eqref{SampledContract}. Then, for a polynomial degree $N\geq 1$, we utilize Theorem~\ref{Prop_Design} to construct a discrete-time system $\bm{\Sigma}_d$ such that $\bm{\Sigma}_c$ is an $N$-th order interpolator of $\bm{\Sigma}_d$ with respect to $\tau$. Based on these discrete-time system $\bm{\Sigma}_d$ and contract $\mathcal{C}_d$, we will derive a condition that ensures the implication \eqref{ContractImplication}.

For an initial condition $x_0\in\mathbb{R}^n$, suppose that $\bm{\Sigma}_d \models \mathcal{C}_d$. This, by Definition~\ref{Def_Implementation}, implies the existence of a control sequence $u_d\in\operatorname{sig}(\mathcal{A}_d)$ such that $\restr{x_d(\cdot;x_0,u_d)}{[0,\frac{T}{\tau}]_{\mathbb{Z}}} \in \operatorname{sig}(\mathcal{G}_d)$. Choosing a continuous-time control input $u_c\in\mathbb{I}_N^\tau(u_d)$ that yields \eqref{InterpolatingCondition}, we immediately conclude from Theorem~\ref{PropTrajectorySetCharacterization} that $(u_c,\restr{x_c(\cdot;x_0,u_c)}{[0,T]_{\mathbb{R}}}) \in \mathcal{T}_c(x_0,u_d)$. It thus follows from \eqref{SetTc} that for all $k=0,1,\ldots,\frac{T}{\tau}-1$, there exist matrices $X^k$ and $U^k$ such that \eqref{AlgorithmTechnical1} holds and
	\begin{align}
		 u_c(k\tau+t) &= u_d(k)\phi_0(t) + U^k\Phi(t), \label{SpectralRep}\\
		x_c(k\tau+t; x_0,u_c) &= x_d(k;x_0,u_d)\phi_0(t)+X^k\Phi(t),\nonumber
	\end{align}
	for all $t\in[0,\tau]_{\mathbb{R}}$. We exploit the representation \eqref{SpectralRep} to derive a condition which guarantees that $u_c(k\tau+t) \in \mathcal{A}_c(k\tau+t)$ and $x_c(k\tau+t; x_0,u_c) \in \mathcal{G}_c(k\tau+t)$ for all $k = 0,1,\ldots, \frac{T}{\tau}-1$ and all $t\in[0,\tau]_{\mathbb{R}}$. This then indicates that $u_c\in\operatorname{sig}(\mathcal{A}_c)$ and that \eqref{ContinuousTimeCondition} holds, which, by Definition~\ref{Def_Implementation}, implies that $\bm{\Sigma}_c\models\mathcal{C}_c$. 

We derive this condition by making use of the so-called \emph{Bernstein basis polynomials} \cite{lorentz2012bernstein}, which are defined as
\begin{equation*}
	\mathfrak{b}_{j,N}(t) = \frac{N!}{j!(N-j)!}t^j(1-t)^{N-j}, \quad j=0,1,\ldots,N.
\end{equation*} 
These polynomials are non-negative over the interval $[0,1]_{\mathbb{R}}$ (\textit{i.e.,} $\mathfrak{b}_{j,N}(t) \geq 0$ for all $t\in[0,1]_{\mathbb{R}}$) and they form a partition of unity, since $\sum_{j=0}^{N}\mathfrak{b}_{j,N}(t) = \big(t + (1-t)\big)^N = 1$. Importantly, the Bernstein basis polynomials are linearly independent over the interval $[0,1]_{\mathbb{R}}$ and, therefore, form a basis for the space of polynomials  up to degree $N$ over this interval.

We now define the vector-valued polynomial function $\mathfrak{B}:[0,1]_{\mathbb{R}}\rightarrow\mathbb{R}^{N+1}$ such that
\begin{equation*}
	\forall t\in[0,1]_{\mathbb{R}}: \;\; \mathfrak{B}(t) = \begin{bmatrix}
		\mathfrak{b}_{0,N}(t)&\mathfrak{b}_{1,N}(t)&\cdots&\mathfrak{b}_{N,N}(t)
	\end{bmatrix}^\top.
\end{equation*}
Bearing in mind that Bernstein basis polynomials form a basis for the polynomial space (up to degree $N$) over the interval $[0,1]_{\mathbb{R}}$, we immediately conclude from \eqref{SpectralRep} that for any $k=0,1,\ldots,\frac{T}{\tau}-1$, there exist matrices $V^k\in\mathbb{R}^{m\times (N+1)}$ and $W^k\in\mathbb{R}^{n\times (N+1)}$ such that 
\begin{equation}\label{BersnteinRepresentation}
	\begin{aligned}
		u_c(k\tau+t) &= V^k\mathfrak{B}\left(\frac{t}{\tau}\right), x_c(k\tau+t;x_0,u_c) = W^k\mathfrak{B}\left(\frac{t}{\tau}\right),
	\end{aligned}
\end{equation}
for all $t\in[0,\tau]_{\mathbb{R}}$. Then, by recalling that Bernstein basis polynomials are non-negative and that they form a partition of unity, we observe that $u_c(k\tau+t)$ and $x_c(k\tau+t;x_0,u_c)$ can be written as \emph{convex} combinations of the columns in matrices $V^ k$ and $W^ k$, respectively. Considering that for all $k = 0,1,\ldots, \frac{T}{\tau}-1$ and all $t\in[0,\tau]_{\mathbb{R}}$, the sets $\mathcal{A}_c(k\tau+t)$ and $\mathcal{G}_c(k\tau+t)$ are convex, we can utilize these matrices $V^ k$ and $W^k$ to ensure that $u_c\in\operatorname{sig}(\mathcal{A}_c)$ and that \eqref{ContinuousTimeCondition} holds. 

After constructing the matrix 
\begin{equation}\label{MMatrix}
	\mathfrak{M} = \begin{bmatrix}
		\mathfrak{B}\left(\frac{t_{N,0}}{\tau}\right)&\mathfrak{B}\left(\frac{t_{N,1}}{\tau}\right)&\cdots&\mathfrak{B}\left(\frac{t_{N,N}}{\tau}\right)
	\end{bmatrix},
\end{equation}
we note from Lemma~\ref{TechnicalLemma} in Appendix~\ref{Sec_TechnicalLemma} that $\mathfrak{M}$ is invertible. This, together with \eqref{SpectralRep} and \eqref{BersnteinRepresentation}, immediately gives  $V^ k = [u_d(k)\;\;U^k]\mathfrak{M}^{-1}$ and $W^k =[x_d(k;x_0,u_d)\;\; X^k]\mathfrak{M}^{-1}$.
Based on this, we now obtain the condition that guarantees $\bm{\Sigma}_c\models\mathcal{C}_c$, see Appendix~\ref{ProofTh_Implication} for the proof. 
\begin{theorem}\label{Th_Implication}
	Consider the continuous-time system $\bm{\Sigma}_c$ and let the contract $\mathcal{C}_c = (\mathcal{A}_c,\mathcal{G}_c)$ be such that Assumption~\ref{Assumption} holds. For the sampling time \eqref{SamplingTime}, let the contract $\mathcal{C}_d = (\mathcal{A}_d,\mathcal{G}_d)$ be given by \eqref{SampledContract}. For a polynomial degree $N\geq 1$, let the discrete-time system $\bm{\Sigma}_d$ be such that $\bm{\Sigma}_c$ is an $N$-th order interpolator of $\bm{\Sigma}_d$ with respect to $\tau$. For an initial condition $x_0\in\mathbb{R}^n$, the system $\bm{\Sigma}_c$ implements the contract $\mathcal{C}_c$ if there exists a control sequence $u_d\in\operatorname{sig}(\mathcal{A}_d)$ such that
	\begin{equation}\label{DiscCond}
		\restr{x_d(\cdot;x_0,u_d)}{[0,\frac{T}{\tau}]_{\mathbb{Z}}} \in \operatorname{sig}(\mathcal{G}_d),
	\end{equation}
	 and for all $k\in[0,\frac{T}{\tau}-1]_{\mathbb{Z}}$, we have that 
	 \begin{equation}\label{InclusionCond}
	 \begin{aligned}
	 		\operatorname{col}\left(\begin{bmatrix}
	 		u_d(k)&U^k
	 	\end{bmatrix}\mathfrak{M}^{-1}\right) &\subset \mathcal{A}_d(k), \\
	 	 \operatorname{col}\left(\begin{bmatrix}
	 		x_d(k;x_0,u_d)&X^k
	 	\end{bmatrix}\mathfrak{M}^{-1}\right) &\subset \mathcal{G}_d(k),
	 \end{aligned}
	 \end{equation}
	where $U^k$ and $X^k$ are given by \eqref{AlgorithmTechnical1}. For any such sequence $u_d$, a continuous-time control input $u_c\in\operatorname{sig}(\mathcal{A}_c)$ that establishes \eqref{ContinuousTimeCondition} is given by 
		\begin{equation}\label{MidU}
			u_c(k\tau+t) = u_d(k)\phi_0(t) + U^k\Phi(t).
	   \end{equation}
	for all $k\in[0,\frac{T}{\tau}-1]_{\mathbb{Z}}$ and all $t\in[0,\tau]_{\mathbb{R}}$.
\end{theorem}

By establishing the implication \eqref{ContractImplication}, Theorem~\ref{Th_Implication} characterizes the continuous-time contract implementation problem in terms of a discrete-time one. This basically allows for the implementation of a continuous-time contract in discrete time rather than in continuous time. In fact, it follows from Theorem~\ref{Th_Implication} that in order to verify that $\bm{\Sigma}_c\models\mathcal{C}_c$, it suffices to ensure that $\bm{\Sigma}_d\models\mathcal{C}_d$ with a control sequence that satisfies \eqref{InclusionCond}, which can be accomplished by the many existing discrete-time schemes. For a solution obtained by any of these methods (\textit{i.e.,} for any discrete-time control that satisfies \eqref{DiscCond} and \eqref{InclusionCond}), Theorem~\ref{Th_Implication} also gives the continuous-time control input subject to which $\bm{\Sigma}_c$ implements $\mathcal{C}_c$. 

\begin{remark}
	Theorem~\ref{Th_Implication} characterizes the continuous-time input $u_c\in\operatorname{sig}(\mathcal{A}_c)$ that establishes \eqref{ContinuousTimeCondition} in terms of a discrete-time input $u_d\in\operatorname{sig}(\mathcal{A}_d)$ that satisfies \eqref{DiscCond} and \eqref{InclusionCond}. The existence of such discrete-time input, however, depends on the choice of the sampling time $\tau$ and the polynomial degree $N$. Intuitively, it follows from \eqref{SamplingTime} and \eqref{SampledContract} that contract discretization with respect to a smaller sampling time $\tau$ results in the construction of discrete-time contracts whose assumptions are more permissive (\textit{i.e.,} they impose fewer limitations on the control input), whereas their guarantees are less strict (\textit{i.e.,} they prescribe larger safe sets). On the other hand, it follows from \eqref{AlgorithmTechnical1} that for a smaller polynomial degree $N$, the constraint \eqref{InclusionCond} is less restrictive in the sense that it entails fewer inclusion conditions. We therefore observe that smaller choices of $\tau$ and $N$ improve the feasibility of the problem. 
\end{remark}
\section{Numerical Example}\label{Sec_NumericalExample}
We consider a robotic agent $\bm{\Sigma}_c$ whose planar motion is described by 
\begin{equation*}
	\begin{aligned}
		\dot{x}(t)&=v_x(t), &\dot{y}(t)&=v_y(t)\\
		\dot{v}_x(t) &= F_{x},&
		\dot{v}_y(t) &= F_{y},
	\end{aligned}
\end{equation*}
where the variables $x$, $v_x$, and $F_{x}$ represent displacement, velocity, and acting force in the longitudinal direction, whereas the variables $y$, $v_y$, and $F_{y}$ represent the corresponding quantities in the lateral direction. We suppose that the robot is initially at rest and located at $(1,1)$, \textit{i.e.,} we have that $x(0) = 1$, $v_x(0) = 0$, $y(0)=1$, and $v_y(0) = 0$. 

We now consider the continuous-time contract $\mathcal{C}_c = (\mathcal{A}_c,\mathcal{G}_c)$ such that 
\begin{equation*}
	\begin{aligned}
		\mathcal{A}_c(t) &= \{(F_{x},F_{y}) \vert F_x,F_y\in[-5,5]_{\mathbb{R}}\}, &0\leq t\leq 5,\\
		\mathcal{G}_c(t) &= \mathcal{X}(t)\times \mathcal{V}(t), &0\leq t\leq 5,
	\end{aligned}
\end{equation*}
where 
\begin{equation*}
	\begin{aligned}
		\mathcal{X}(t)\! &=\! \left\{\begin{aligned}
			&\!\{(x,y)\vert x\!\in\![0,1.5]_{\mathbb{R}}, y\!\in\![0,1.5]_{\mathbb{R}}\}, &0\leq t< 1,\\
			&\!\{(x,y)\vert x\!\in\![0.5, 2]_{\mathbb{R}}, y\!\in\![0.5,1.5]_{\mathbb{R}}\},  &1\leq t< 2,\\
			&\!\{(x,y)\vert x\!\in\![1,2.5]_{\mathbb{R}}, y\!\in\![-1,3]_{\mathbb{R}}\},   &2\leq t< 3,\\
			&\!\{(x,y)\vert x\!\in\![-1,1.5]_{\mathbb{R}}, y\!\in\![-1,0.75]_{\mathbb{R}}\},  \!\!\!&3\leq t< 4,\\
			&\!\{(x,y)\vert x\!\in\![-2,1.25]_{\mathbb{R}},y\!\in\![-2,0]_{\mathbb{R}}\},  &4\leq t\leq 5,
		\end{aligned}\right.\\
	\mathcal{V}(t) \!&= \! \{(v_x,v_y)\vert v_x\in[-2,2]_{\mathbb{R}},v_y\in[-2,2]_{\mathbb{R}}\}, \quad \!0\leq t\leq 5.	
	\end{aligned}
\end{equation*}
We note that the contract $\mathcal{C}_c$ satisfies the Assumption~\ref{Assumption}. 

Our objective is to establish $\bm{\Sigma}_c\models\mathcal{C}_c$. We accomplish this by making use of the proposed vertical contract implementation framework.

 By defining $x_c = (x,y,v_x,v_y)^\top$ and $u_c = (F_x,F_y)^\top$, we formulate $\bm{\Sigma}_c$ as \eqref{ContinuousSystem}. We now discretize the contract $\mathcal{C}_c$ and the system $\bm{\Sigma}_c$ according to the procedure described in Section~\ref{Sec_Discretization}. For this purpose, we recall that $\mathcal{C}_c$ satisfies the Assumption~\ref{Assumption} and conclude that \eqref{Smoothness} holds. After observing that $\munderbar{r}_T = 3$, we obtain $r_c = \min_{t\in[0,2]_{\mathbb{R}}}r_t = 3$. Bearing in mind that $\lceil \frac{T}{r_c} \rceil = 2$, we take $\ell_d = 5$ and accordingly use \eqref{SamplingTime} to select the sampling time $\tau = 1$. We then construct the discrete-time contract $\mathcal{C}_d = (\mathcal{A}_d,\mathcal{G}_d)$ as in \eqref{SampledContract}.
 We subsequently discretize the system $\bm{\Sigma}_c$ by choosing $N = 5$ and utilizing Theorem~\ref{Prop_Design} to construct the discrete-time system $\bm{\Sigma}_d$ such that $\bm{\Sigma}_c$ is an $N$-th order interpolator ot $\bm{\Sigma}_d$ with respect to $\tau$. 
 
 Having constructed the discrete-time contract $\mathcal{C}_d$ and the system $\bm{\Sigma}_d$, we proceed with translating the continuous-time implementation problem into the discrete-time one. We therefore apply Theorem~\ref{Th_Implication} to construct a control sequence $u_d\in\operatorname{sig}(\mathcal{A}_d)$ such that \eqref{DiscCond} and \eqref{InclusionCond} hold for all $k\in[0,\frac{T}{\tau}-1]_{\mathbb{Z}}$. For this purpose, we adopt the approach taken in \cite{belta2019formal} to encode characterize such control sequence $u_d$ as the solution of a linear program. We use the optimization toolbox of MATLAB 2024b to solve this linear program. Lastly, we use \eqref{MidU} to construct the continuous-time control input $u_c\in\operatorname{sig}(\mathcal{A}_c)$ that achieves \eqref{ContinuousTimeCondition}, which, by Definition~\ref{Def_Implementation}, indicates that $\bm{\Sigma}_c\models\mathcal{C}_c$.
 
 The displacement, velocity, and the acting force of the robot are depicted in Figure~\ref{Result}. The permissible range for the acting force (dictated by the assumptions $\mathcal{A}_c$) and the safe boundaries for robot operation (prescribed by the guarantees $\mathcal{G}_c$) are illustrated by colored patches. It is clear from Figure~\ref{Result} that the continuous-time control input $u_c$, obtained as in \eqref{MidU}, complies with the limitations dictated by the assumptions $\mathcal{A}_c$, while it enforces the robot to operate within the safe boundaries prescribed by the guarantees $\mathcal{G}_c$.    
 \begin{figure}
 	\centering
 		\includegraphics[width=0.5\textwidth]{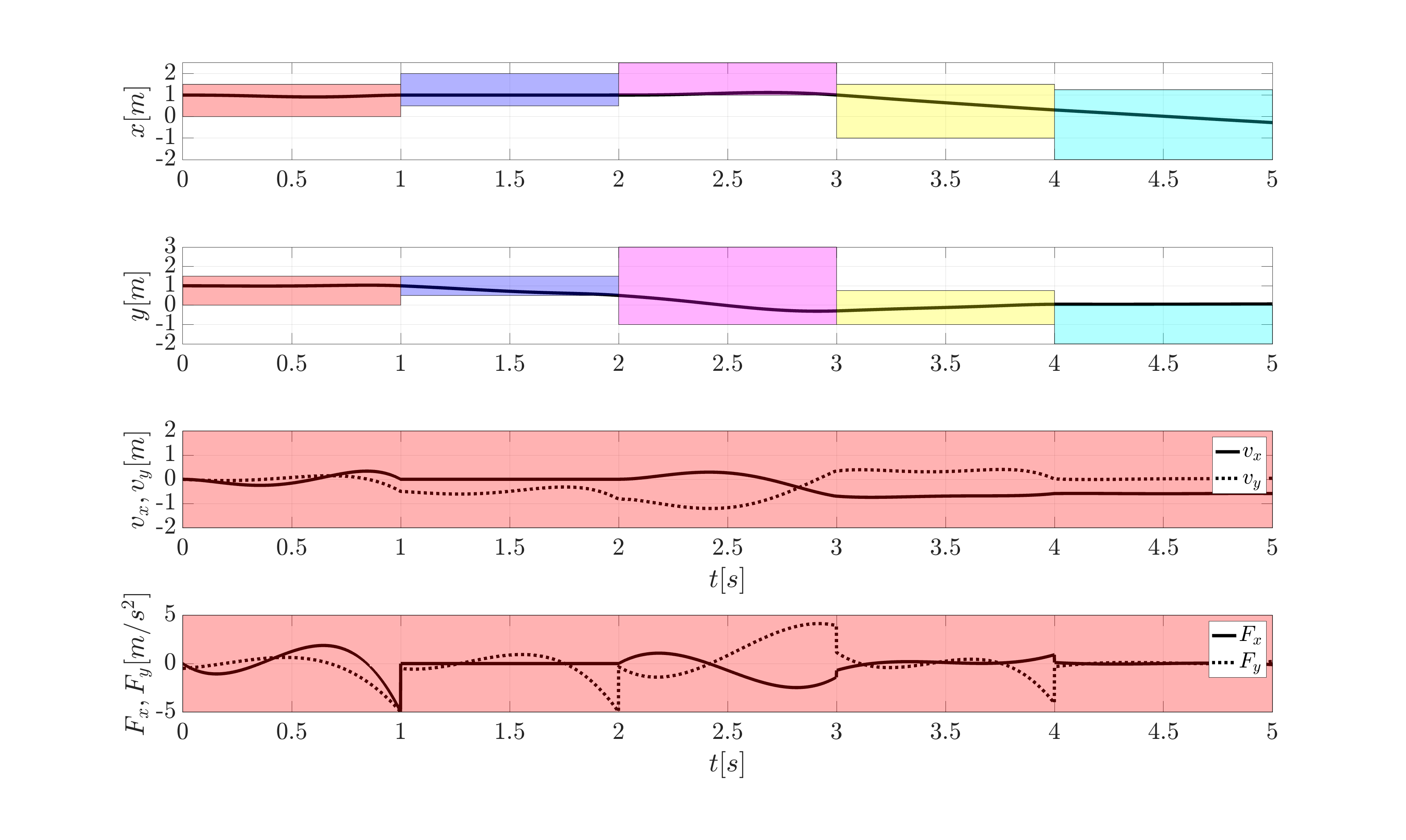}
 		\caption{The continuous-time control input obtained as in \eqref{MidU} satisfies the assumptions $\mathcal{A}_c$ and enforces the robot to operate in accordance with the guarantees $\mathcal{G}_c$. }
 		\label{Result}
 \end{figure}
 \section{Conclusion}\label{Sec_Conclusion}
 In this paper, we adopted the contract formalism to develop a framework that enables translation of the safety control problem from continuous to discrete time. We constructed contracts as pairs of assumptions and guarantees, which are set-valued mappings that respectively describe limitations on the applied control input and safe boundaries within which the system must operate. We then formalized safety control as contract implementation, which determines the feasibility of safety enforcement (as prescribed by guarantees) despite design limitations (as imposed by assumptions). We then established implementation according to a vertical hierarchy, where we formulated the continuous-time implementation problem in terms of a discrete-time one. To accomplish this, we first discretized the continuous-time contract with respect to a suitable sampling time. We then used the notion of system interpolation to construct a discrete-time model such that input-state trajectories of the continuous-time system can be realized as piecewise polynomial interpolations of the input-state trajectories of the discrete-time model. Based on this realization, we derived a condition that guarantees feasibility of the continuous-time implementation problem in terms of the discrete-time one. Lastly, we constructed the control input that yields implementation in continuous time on the basis of the control sequence that achieves implementation in discrete time. 
\appendix
\subsection{Technical Lemma}\label{Sec_TechnicalLemma}
\begin{lemma}\label{TechnicalLemma}
	The matrix $\mathfrak{M}$, defined in \eqref{MMatrix}, is invertible.
\end{lemma}
\begin{proof}
	To prove the invertibility of $\mathfrak{M}$, we use contradiction to show that $\mathfrak{M}$ is non-singular. We therefore suppose that $\mathfrak{M}$ is singular. This implies the existence of a non-zero vector $c \in \mathbb{R}^{N+1}$ such that $c^\top\mathfrak{M} = 0$. We let $c_0,c_1,\ldots,c_N\in\mathbb{R}$ be such that $c = (c_0,c_1,\cdots,c_N)^\top$ and define the polynomial $q\in\restr{\mathbb{P}_N^1}{[0,\tau]_{\mathbb{R}}}$ such that 
	\begin{equation}\label{Technical1}
		\forall t\in[0,\tau]_{\mathbb{R}}: \quad q(t) = c^ \top\mathfrak{B}\left(\frac{t}{\tau}\right).
	\end{equation}
	This, as a consequence of fact that $c^ \top\mathfrak{M} = 0$, implies that 
	\begin{equation}\label{Technical2}
		\forall j\in[0,N]_{\mathbb{Z}}: \quad q(t_{N,j}) = 0.
	\end{equation} 
	We now consider the polynomial function \eqref{PhiFunction} and recall that $\phi_i(t_{N,i}) = 1$, whereas $\phi_i(t_{N,j}) = 0$ for all $j\neq i$. This indicates that polynomials $\phi_0,\phi_1,\ldots,\phi_N$ are linearly independent and, therefore, form a basis for $\restr{\mathbb{P}_N^{1}}{[0,\tau]_{\mathbb{R}}}$. As a consequence, by considering that $q\in\restr{\mathbb{P}_N^1}{[0,\tau]_{\mathbb{R}}}$, we immediately conclude that  
	\begin{equation*}
		\forall t\in[0,\tau]_{\mathbb{R}}: \quad q(t) = \sum_{j=0}^{N}q(t_{N,j}) \phi_j(t),
	\end{equation*}
	which, as a result of \eqref{Technical2}, implies that 
		\begin{equation*}
		\forall t\in[0,\tau]_{\mathbb{R}}: \quad q(t) = 0.
	\end{equation*}
	This, together with \eqref{Technical1}, results in
	\begin{equation*}
		\forall t\in[0,\tau]_{\mathbb{R}}: \quad c^ \top\mathfrak{B}\left(\frac{t}{\tau}\right) = 0,
	\end{equation*}
	which, due to the linear independence of the Bernstein basis polynomials, implies that $c = 0$, contradicting the initial assumption that $c$ is non-zero. We thus conclude, by contradiction, that $\mathfrak{M}$ is non-singular and, therefore, invertible.  
\end{proof}
\subsection{Proof of Theorem~\ref{Th_Implication}}\label{ProofTh_Implication}
Suppose that there exists a control sequence $u_d\in\operatorname{sig}(\mathcal{A}_d)$ such that \eqref{DiscCond} and \eqref{InclusionCond} hold for all $k\in[0,\frac{T}{\tau}-1]_{\mathbb{Z}}$. We let $u_c:[0,T]\rightarrow\mathbb{R}^m$ be such that \eqref{MidU} holds for all $k\in[0,\frac{T}{\tau}-1]_{\mathbb{Z}}$ and all $t\in[0,\tau]_{\mathbb{R}}$. We show that $u_c\in\operatorname{sig}(\mathcal{A}_c)$ and that \eqref{ContinuousTimeCondition} holds, which, by Definition~\ref{Def_Implementation}, implies that $\bm{\Sigma}_c\models\mathcal{C}_c$.

Considering \eqref{MidU}, we immediately conclude from Theorem~\ref{PropTrajectorySetCharacterization} that for any $k=0,1,\ldots,\frac{T}{\tau}-1$, we have that
	\begin{equation}\label{MidX}
		x_c(k\tau+t; x_0,u_c) = x_d(k;x_0,u_d)\phi_0(t)+X^k\Phi(t)
	\end{equation}
	for all $t\in[0,\tau]_{\mathbb{R}}$. We now show that $u_c\in\operatorname{sig}(\mathcal{A}_c)$ and that $\restr{x_c(\cdot;x_0,u_c)}{[0,T]_{\mathbb{R}}} \in \operatorname{sig}(\mathcal{G}_c)$. For this purpose, we observe from \eqref{MidU} and \eqref{MidX} that there exist matrices $V^k\in\mathbb{R}^{m\times (N+1)}$ and $W^k\in\mathbb{R}^{n\times (N+1)}$ such that  \eqref{BersnteinRepresentation} holds. Since $\mathcal{A}_d(k)$ and $\mathcal{G}_d(k)$ are convex sets, we immediately conclude from \eqref{InclusionCond}
	that for any $k=0,1,\ldots,\frac{T}{\tau}-1$, we have that 
	\begin{equation*}
		\forall t\in[0,\tau]_{\mathbb{R}}: \; u_c(k\tau+t) \in \mathcal{A}_d(k), \; x_c(k\tau+t; x_0,u_c) \in \mathcal{G}_d(k),
	\end{equation*}
	where we have used the fact that Bernstein polynomials are non-negative and form a partition of unity.
	 This, as a consequence of \eqref{SampledContract}, implies that for all $t\in[0,\tau]_{\mathbb{R}}$, we have that 
	\begin{equation*}
		\begin{aligned}
			u_c(k\tau+t) \in \mathcal{A}_c(k\tau+t), \quad x_c(k\tau+t; x_0,u_c) \in \mathcal{G}_c(k\tau+t),
		\end{aligned}
	\end{equation*}
	where we have also used the fact that $\mathcal{A}_d(k) \neq \emptyset$ and $\mathcal{G}_d(k) \neq \emptyset$ for all $k\in [0,\frac{T}{\tau}]_{\mathbb{Z}}$. This implies that $u_c\in\operatorname{sig}(\mathcal{A}_c)$ and that \eqref{ContinuousTimeCondition} holds, which, by Definition~\ref{Def_Implementation}, indicates that $\bm{\Sigma}_c\models\mathcal{C}_c$.
\bibliographystyle{IEEEtran}
\bibliography{Reference.bib}

\end{document}